\documentclass[]{birkmult}

\usepackage{color}
\usepackage{soul,cancel,url}

 \newtheorem{theorem}{Theorem}[section]
 \newtheorem{corollary}[theorem]{Corollary}
 \newtheorem{lemma}[theorem]{Lemma}
 
 \theoremstyle{definition}
 
 \theoremstyle{remark}
 \newtheorem{remark}[theorem]{Remark}

\numberwithin{equation}{section}
\renewcommand{\Re}{\operatorname{Re}}
\renewcommand{\Im}{\operatorname{Im}}
\newcommand{\tr}{\operatorname{tr}}
\newcommand{\eq} [1] {\begin{equation}\label{#1}\quad}
\newcommand{\en} {\end{equation}}
\newcommand{\scal}[1]{\langle#1\rangle}
\newcommand{\norm}[1]{\left\Vert#1\right\Vert}

\newcommand{\abs}[1]{\left\vert#1\right\vert}
\newcommand{\fl}[1]{\left\lfloor#1\right\rfloor}
\newcommand{\diag}{\operatorname{diag}}
\newcommand{\Ker}{\operatorname{Ker}}

\newcommand{\C}{\mathbb{C}}
\newcommand{\R}{\mathbb{R}}

\begin{document}
\title[Numerical ranges  of companion matrices]
{Numerical ranges of companion matrices:\\ flat portions on the boundary}

\author[Eldred]{Jeffrey  Eldred}
\address{Department of Physics\\ Indiana University\\ Bloomington, IN 47408}

\email{jseldred@email.wm.edu, jseldred@indiana.edu}

\author[Rodman]{Leiba Rodman}
\address{Department of Mathematics\\ College of William and Mary\\
Williamsburg, VA 23187\\ USA}

\email{lxrodm@math.wm.edu}

\author[Spitkovsky]{Ilya M. Spitkovsky}
\address{Department of Mathematics\\ College of William and Mary\\
Williamsburg, VA 23187\\ USA}

\email{ilya@math.wm.edu, imspitkovsky@gmail.com}




\begin{abstract}
Criterion for  a companion matrix to have a certain number of flat portions
on the boundary of its numerical range is given. The criterion is specialized to the cases of
$3\times 3$ and $4\times 4$ matrices. In the latter case, it is proved that
a $4\times 4$ unitarily irreducible companion matrix cannot have $3$ flat portions on the boundary of its
numerical range. Numerical examples are given to illustrate the main results. \end{abstract}

\maketitle

\section{Introduction}
The numerical range $W(A)$ of an $n\times n$ matrix $A$ is a subset
of the complex plane $\C$ defined as \[ W(A)=\{ \scal{Ax,x}\colon x\in
\C^n, \ \norm{x}=1\},
\]
where $\scal{\cdot, \cdot}$ stands for the standard inner product in $\C^n$.
This set first appeared in classical works by Toeplitz \cite{Toe18} and
Hausdorff \cite{Hau},  and since then has been studied intensively.
Among standard contemporary references are \cite{GusRa} and
\cite[Chapter I]{HJ2}, and all properties of the numerical range we will
be using without proof can be found, e.g., in these two monographs.

Among other things, it is of interest to locate flat portions (if any) on
the boundary $\partial W(A)$ of the numerical range, and in particular
to establish a bound for the number $f(A)$ of such portions for various
matrix classes. If $A$ is {\em unitarily reducible}, that is, unitarily
similar to a block diagonal matrix with at least two diagonal blocks
$A_j$, then $W(A)$ is the convex hull of $W(A_j)$. The flat portions on
$\partial W(A)$ are then bound to emerge, unless one of $W(A_j)$
contains all others. In particular, for normal $A$ the blocks $A_j$ can
be made one-dimensional and $W(A)$ is nothing but the convex hull
of the spectrum $\sigma(A)$.  It is easy to see therefore that $f(A)$ is at
most $n$ for normal $A$.

The picture is trivial for $n=2$: $f(A)=0$ if $A$ is not normal, since
$W(A)$ is then an elliptical disk, $f(A)=1$ for normal $A$ different from
a scalar multiple of the identity, since $W(A)$ is then a line segment,
and $f(\lambda I)=0$. For $n=3$, the classification of possible shapes
of $W(A)$ was given by Kippenhahn (\cite{Ki}, see also more
accessibly \cite{Ki08}). From this classification it easily follows that the
maximal possible $f(A)$ is actually attained by normal $A$, while
$f(A)$ is at most two for non-normal unitarily reducible matrices, and
at most one for unitarily irreducible ones. Constructive descriptions of
$3\times 3$ matrices matrices $A$ with flat portions on $\partial
W(A)$ were obtained in \cite{KRS,RS05}.

The case $n=4$ was undertaken in \cite[Theorem 37]{BS041}, where it
was established the bound 4 is then sharp, while for unitarily
irreducible $4\times 4$ matrix $A$, the number $f(A)$ is at most 3. On
the other hand, for any $n$ there exist $n\times n$ unitarily reducible
matrices $A$ for which $f(A)=2(n-2)$, see Example 38 in \cite{BS041}
(suggested to the authors by C.-K. Li). It is not known, when $n>4$, (i)
whether this delivers the sharp upper bound for $f(A)$ (note that
$2(n-2)=4$ for $n=4$) and (ii) what is the upper bound for unitarily
irreducible $A$.

In this paper we focus on the case when $A$ is a {\em companion
matrix}, that is, \eq{A}
A=\left[\begin{matrix} 0 & 1 &  & \\ & \ddots & \ddots & \\ & & 0 & 1 \\
-a_0 & \ldots & -a_{n-2} & -a_{n-1}
\end{matrix}\right]. \en It is well known that the elements of the last row of (\ref{A})
coincide, up to the sign, with the coefficients of its characteristic
polynomial: \eq{char} \det (A- \lambda
I)=\lambda^n+a_{n-1}\lambda^{n-1}+\cdots+a_0.\en These matrices
were treated in \cite{GauWu07}, where in particular it was established
that for a companion $n\times n$ matrix $A$, $f(A)\leq n$ and all
matrices $A$ with $f(A)=n$ were described. They happen to be
unitarily reducible, and the question of the maximal number of flat
portions for unitarily irreducible companion matrices also remains
open.

In our paper, we further tackle the issue of flat portions
on $\partial W(A)$ for companion matrices $A$. Necessary and sufficient conditions
for such portions to exist are described in Section~2. For arbitrary $n$ they are rather cumbersome, and
(at least in their sufficient part) not easy to check. However, for $n=3,4$ they can be recast into constructively
verifiable criteria, allowing in particular to describe all possible values of $f(A)$. The cases $n=3$ and $n=4$ are treated in Sections 3 and 4, respectively.

\section{Conditions for flat portions existence}

 For convenience of reference, we start with two
statements applicable to arbitrary $n\times n$ matrices $A$. Recall
that $\Re A=\frac{1}{2}(A+A^*)$ and $\Im A=\frac{1}{2i}(A-A^*)$.

\begin{lemma}\label{l:cri}Let $A$ be an $n\times n$ matrix $A$. Then $\partial W(A)$ contains
a vertical flat portion to the right of $W(A)$ if and only if \\ {\em (i)} the
maximal eigenvalue  $\lambda_{\max}$ of $\Re A$ is not simple, and
\\ {\em (ii)} the compression of $\Im A$ (equivalently, of $A$) onto the eigenspace $\frak L$
of $\Re A$ corresponding to $\lambda_{\max}$ is not a scalar
multiple of the identity. \end{lemma} This lemma is well known and
was used, e.g., in \cite{KRS,BS041}.

Formally speaking, (i) follows from (ii), but we prefer (i) to be stated
explicitly since it is the condition addressed in Theorem \ref{th:nec}
below.

\begin{lemma}\label{l:red}Let $A$ be an $n\times n$ matrix $A$. Suppose that $\Re A$
has an eigenvalue $\lambda$ of multiplicity bigger that $\fl{n/2}$
while the compression of $\Im A$ onto the corresponding eigenspace
$\frak L$ of $\Re A$ is a scalar multiple of the identity. Then $A$ is
unitarily reducible.\end{lemma} \begin{proof} Passing from $A$ to
$A-zI$ with an appropriate choice of $z\in\C$, we may without loss of
generality suppose that $\lambda=0$ and the compression of $\Im A$
onto $\frak L:=\Ker\Re A$ is the zero operator. But the latter condition
means simply that $(\Im A)\frak L\perp\frak L$.  Since $2\dim{\frak
L}>n$, this is only possible if $\Im A$ is not injective on $\frak L$, that
is, $\frak L$ contains a non-zero vector $x$ from $\Ker\Im A$. Then $x$
is an eigenvector for both $\Re A$ and $\Im A$ (equivalently, for both
$A$ and $A^*$), which makes it a normally splitting eigenvector for
$A$, and $A$ itself --- unitarily reducible into $1\times 1$ and
$(n-1)\times (n-1)$ blocks.
\end{proof}

\begin{remark}\label{Rem}For $n=3$ condition on
$\lambda$ in Lemma ~\ref{l:red} merely means that this is not a
simple eigenvalue. Consequently, for unitarily irreducible $3\times 3$
matrices condition (ii) of Lemma~\ref{l:cri} can be dropped. This
observation was also used in \cite{KRS,BS041}. \end{remark}

For companion matrices, a constructive criterion of unitary reducibility is known. It was obtained in
\cite[Section 1]{GauWu04} and can be summarized as follows.

\begin{lemma}\label{l:redu}An $n\times n$ companion matrix is unitarily reducible if and only if
$\sigma(A)=\{\eta\omega_j\colon j\in J_1\}\cup \{\overline{\eta}^{-1}\omega_j\colon j\in J_2\}$ for some
$\eta\in\C\setminus\{0\}$ and partition $J_1\cup J_2$ of $\{1,\ldots,n\}$, where both $J_1$ and $J_2$ are
non-empty; $\omega_1,\ldots,\omega_n$ being
the set of all $n$th roots of $1$. If this condition holds, then $A$ is unitarily similar to $A_1\oplus A_2$, with
$\sigma(A_1)=\{\eta\omega_j\colon j\in J_1\}$, $\sigma(A_2)=\{\overline{\eta}^{-1}\omega_j\colon j\in J_2\}$. The matrix $A$ is unitary if $\abs{\eta}=1$, and $A_1, A_2$ are unitarily irreducible otherwise. \end{lemma}

We are now ready to state the necessary condition for the flat portion existence.

\begin{theorem}\label{th:nec}Let $A$ be given by {\em (\ref{A})}. Then for $W(A)$
to have a flat portion on the boundary it is necessary that \eq{cond1}
\sum_{j=0}^{n-2}a_j\omega^{n-j}\sin\frac{\pi(j+1)}{n}=\sin\frac{\pi}{n}
\en and \eq{cond2}\Re(a_{n-1}\omega)=
\sum_{j=2}^{n-1}\frac{\abs{\gamma_j}^2}{\cos\frac{\pi}{n}-\cos\frac{\pi
j}{n}}-\cos\frac{\pi}{n} \en for some $\omega$ with $\abs{\omega}=1$
and\eq{gamma} \gamma_j=\frac{1}{\sqrt{2n}}\left(\sin\frac{\pi
j(n-1)}{n}-\sum_{k=0}^{n-2}a_{k}\omega^{n-k}\sin\frac{\pi
j(k+1)}{n}\right). \en If conditions {\em (\ref{cond1}), (\ref{cond2})}
hold, then the potential flat portion passes through the point
$\overline{\omega}\cos\frac{\pi}{n}$ and has the slope
$\frac{\pi}{2}-\arg\omega$. \end{theorem}
\begin{proof}Observe first of all that for any $\omega$ with absolute value one the
matrix $\omega A$, while not being companion itself (for
$\omega\neq 1$), is nevertheless unitarily similar to a companion
matrix \[
B=\left[\begin{matrix} 0 & 1 &  & \\ & \ddots & \ddots & \\ & & 0 & 1 \\
-b_0 & \ldots & -b_{n-2} & -b_{n-1}
\end{matrix}\right], \]
where  $b_j=a_j\omega^{n-j}$: $ \omega A= \Omega^{-1}B\Omega$
with \eq{Omega} \Omega=\diag [1,\omega,\ldots,\omega^{n-1}].\en
Consequently,
\[ W(A)=\overline{\omega}W(\omega A)=\overline{\omega}W(B).\]
It therefore suffices to show that conditions
(\ref{cond1}), (\ref{cond2}) with $\omega=1$ are necessary for
$\partial W(A)$ to contain a vertical line segment located to the right
of $W(A)$, and that this line segment (if it exists) passes through the
real point $\cos\frac{\pi}{n}$.

Let us show that (\ref{cond1}), (\ref{cond2}) can be interpreted as
condition (i) of Lemma~\ref{l:cri} for $A$ given by (\ref{A}).

Due to the interlacing property of eigenvalues of hermitian matrices,
$\lambda_{\max}$ will be the maximal eigenvalue of all $(n-1)\times
(n-1)$ principal submatrices of $A$. For $A$ given by (\ref{A}), \eq{reA}
\Re A=\frac{1}{2} \left[\begin{matrix} T & \vline & \begin{matrix}  -\overline{a_0} \\ \vdots \\
 -\overline{a_{n-3}} \\ 1-\overline{a_{n-2}}\end{matrix}  \\ \hline
\begin{matrix} -a_0 & \ldots & -a_{n-3} & 1-a_{n-2}\end{matrix}  &  \vline & -2\Re a_{n-1}
\end{matrix}\right] ,  \en
where $T$ is the $(n-1)\times (n-1)$  tridiagonal matrix with zeros on
the main diagonal and ones on two side diagonals:
\[ T= \left[\begin{matrix} 0 & 1 & &   \\ 1 & \ddots & \ddots &   \\
 &\ddots  & \ddots  & 1  \\ &  & 1 & 0
\end{matrix}\right]. \] The eigenvalues and the eigenvectors of $T$ are well known. Namely (see, e.g., \cite{HaHa92} or \cite[Section 2.2]{BGru}),
\[ Tv_j=2\cos\frac{\pi j}{n}v_j,\quad j=1,\ldots, n-1, \] where
\begin{equation}\label{may131} v_j=\left[\sin\frac{\pi j}{n},\ldots, \sin\frac{\pi j(n-1)}{n}\right]^T.
\end{equation}  So, the abscissa of the potential vertical flat portion is indeed $\cos\frac{\pi}{n}$.
On the other hand, the left upper $(n-1)\times (n-1)$ block of $A$ is
the Jordan cell $J_{n-1}$, so that $W(A)\supset W(J_{n-1})$. In its turn,
$W(J_{n-1})=\{z\colon \abs{z}\leq \cos\frac{\pi}{n} \}$ (see, e.g.,
\cite{HaHa92}), so that the above mentioned flat portion should be
passing through the real point $\cos\frac{\pi}{n}$.

As it is stated in \cite{HaHa92} (and can also be checked via a routine
trigonometrical calculation), $\norm{v_j}^2=n/2$ for all $j$. Therefore,
the matrix \begin{equation} \label{may132}
V=\sqrt{\frac{2}{n}}\left[\sin\frac{\pi jk}{n}\right]_{k,j=1}^{n-1}
\end{equation} is an hermitian (actually, real symmetric) involution which
diagonalizes $T$: \[ T=2V\diag\left[\cos\frac{\pi}{n}, \ldots,
\cos\frac{\pi(n-1)}{n}\right]V. \] Consequently, matrix (\ref{reA})
is unitarily similar to \eq{H} H=\left[\begin{matrix}\begin{matrix} \cos\frac{\pi}{n} & & \\ & \ddots & \\
& & \cos\frac{\pi(n-1)}{n}\end{matrix} & \vline & \begin{matrix}  \overline{\gamma_1} \\ \vdots \\
 \overline{\gamma_{n-1}} \end{matrix}  \\ \hline
\begin{matrix} \gamma_1 & \ldots & \gamma_{n-1} \end{matrix}  &  \vline & -\Re a_{n-1}
\end{matrix}\right], \en where \[ \gamma_j=\frac{1}{\sqrt{2n}}\left(\sin\frac{\pi
j(n-1)}{n}-\sum_{k=0}^{n-2}a_{k}\sin\frac{\pi j(k+1)}{n}\right) \]
(observe that the latter formula is the particular case of (\ref{gamma})
for $\omega=1$).

From (\ref{H}) it is easily seen that  \[
\det\left(H-\cos\frac{\pi}{n}I\right)=-\abs{\gamma_1}^2\prod_{j=2}^{n-1}\left(\cos\frac{\pi
j}{n}-\cos\frac{\pi}{n}\right). \] Thus, $\cos\frac{\pi}{n}$ is an
eigenvalue of $H$ (and therefore of $\Re A$) if and only if
$\gamma_1=0$. This coincides with (\ref{cond1}) in which
$\omega=1$.

The multiplicity of $\cos\frac{\pi}{n}$ as an eigenvalue of $\Re A$
cannot exceed 2, since the matrix $H-\cos\frac{\pi}{n}I$ contains a
non-singular $(n-2)\times (n-2)$ submatrix \[ \diag\left[\cos\frac{\pi
j}{n}-\cos\frac{\pi}{n}\right]_{j=2}^{n-1}.\] In order for this multiplicity
to equal 2 it is necessary and sufficient that, in addition to
$\gamma_1=0$, the right lower $(n-1)\times (n-1)$ submatrix of
$H-\cos\frac{\pi}{n}I$,
\[ \left[\begin{matrix}\begin{matrix} \cos\frac{2\pi}{n}-\cos\frac{\pi}{n} & & \\ & \ddots & \\
& & \cos\frac{\pi(n-1)}{n}-\cos\frac{\pi}{n}  \end{matrix} & \vline & \begin{matrix}  \overline{\gamma_2} \\ \vdots \\
 \overline{\gamma_{n-1}} \end{matrix}  \\ \hline
\begin{matrix} \gamma_2 & \ldots & \gamma_{n-1} \end{matrix}  &  \vline & -\Re a_{n-1}-\cos\frac{\pi}{n}
\end{matrix}\right], \]
is singular.  This is an arrowhead matrix, the determinant of which can
be computed by an easy induction and equals \[
\left(\sum_{j=2}^{n-1}\frac{\abs{\gamma_j}^2}{\cos\frac{\pi}{n}-\cos\frac{\pi
j}{n}}-\Re a_{n-1}-\cos\frac{\pi}{n}\right)\cdot\prod_{j=2}^{n-1}
\left(\cos\frac{\pi j}{n}-\cos\frac{\pi}{n}\right).\] Thus, it equals zero if
and only if (\ref{cond2}) holds (once again, with $\omega=1$).
\end{proof} Note that necessity of condition (\ref{cond1}) in a slightly
different way was established in \cite{GauWu07}, see Lemma~3 there.

It follows from Lemma~\ref{l:redu} and Theorem~\ref{th:nec} that in the generic case matrices (\ref{A}) are
unitarily irreducible and have no flat portions on the boundary. Namely, the set of companion matrices for
which (\ref{cond1}) has no unimodular solutions is open and dense within the set of all companion matrices.
The openness of this set is clear from continuity of roots of algebraic equations as functions of the
equations' coefficients. As for denseness, assume $a_0\neq 0$, and let
$$ \left(\sum_{j=0}^{n-2}a_j\omega^{n-j}\sin\frac{\pi(j+1)}{n}\right)- \sin\frac{\pi}{n}
=a_0 (\omega - \omega_1)\cdots (\omega - \omega_n),
 $$
where $\omega_1,\ldots ,\omega_n$ are all the roots of (\ref{cond1}) counted with multiplicities.
Note that $\omega_1\cdots \omega_n=(-1)^{n+1}a_0^{-1} \sin\frac{\pi}{n}$ and
$\sum_{j=1}^n \omega_j^{-1}=0$. We now perturb $\omega_1,\ldots, \omega_n$ slightly resulting in
$\omega_1', \ldots \omega_n'$ respectively such that
none of the $\omega_j'$s is unimodular and the equality
$\sum_{j=1}^n (\omega_j')^{-1}=0$ holds. Clearly such perturbation is possible. Now define
$a_0',\ldots ,a_n'$ by the equalities
$$ \omega_1'\cdots \omega_n'=(-1)^{n+1}(a_0')^{-1} \sin\frac{\pi}{n} $$
and
$$ \left(\sum_{j=0}^{n-2}a_j'\omega^{n-j}\sin\frac{\pi(j+1)}{n}\right)- \sin\frac{\pi}{n}
=a_0' (\omega - \omega_1')\cdots (\omega - \omega_n').
 $$
As a result, a companion matrix is obtained, as close as we wish to $A$, for which the corresponding
equation (\ref{cond1}) has no unimodular solutions.

A specific subclass of unitarily irreducible companion matrices with no flat portions on
the boundary of their numerical ranges is delivered by the following

\begin{corollary}\label{c:zeros}Let $A$ be given by {\em (\ref{A})} with \eq{zeros} a_0=\cdots=a_{n-2}=0.\en Then $A$ is unitarily irreducible and $W(A)$ has no flat portions on the boundary.\end{corollary}
Indeed, such $A$ are singular, and therefore (as follows from Lemma~\ref{l:redu}) unitarily irreducible. On the other hand, equation (\ref{cond1}) takes the form $0=\sin\frac{\pi}{n}$ and thus has no solutions.

Note that if, in addition to (\ref{zeros}), also $a_{n-1}=0$, then $A$ is simply a nilpotent Jordan block,
with $W(A)$ being a circular disk. If $a_{n-1}\neq 0$, the numerical range of $A$ cannot be circular according
to \cite[Theorem 1]{GauWu07}, but still there will be no flat portions on $\partial W(A)$. An example
illustrating this, more interesting, situation when $n=4$, will be given in Section~\ref{4}.

Of course,  a criterion for the flat portion existence can be formulated
by imposing
 condition (ii) of Lemma~\ref{l:cri} (interpreted for the case of companion matrices) on matrices satisfying
 Theorem~\ref{th:nec}. \begin{theorem}\label{th:cri}Let conditions
 {\em (\ref{cond1}), (\ref{cond2})} hold for some matrix $A$ given by {\em (\ref{A})} and $\omega$ having
absolute value 1.
 Then $\partial W(A)$ has a flat portion passing through $\overline{\omega}\cos\frac{\pi}{n}$ if and only if
 at least one of the scalar products $\scal{\Im(\omega A)x_1,x_2}$ and $\scal{\Im(\omega A)x_2,x_2}$ differs from zero. Here
 \eq{x} x_1=\Omega^{-1}\left[\begin{matrix}v_1\\ 0\end{matrix}\right], \quad
 x_2=\Omega^{-1}\left[\begin{matrix}V & 0\\ 0 & 1\end{matrix}\right]\xi,
 \en with $\Omega$, $v_1$ and $V$ given by {\em (\ref{Omega})}, {\em (\ref{may131})} and {\em
(\ref{may132})} respectively,
\[ \xi=[0,\xi_2,\ldots,\xi_{n-1},1]^T \text{ and }
\xi_j=\frac{\overline{\gamma_j}}{\cos\frac{\pi}{n}-\cos\frac{\pi j}{n}},
\quad j=2,\ldots, n-1.\] \end{theorem}  \begin{proof}Under conditions of
Theorem~\ref{th:nec} (and in the notation of its proof), vectors
$[1,0,\ldots,0]^T (\in\C^n)$ and $\xi$ form a basis of $\Ker
\left(H-\cos\frac{\pi}{n}I\right)$. Consequently, (\ref{x}) delivers a
basis for ${\frak L}=\Ker \left(\Re A-\cos\frac{\pi}{n}I\right)$. Since \[
\scal{\Im (\omega A)x_1,x_1}=\scal{(\Im J_{n-1})v_1,v_1}=0, \] the
compression of $\Im (\omega A)$ onto $\frak L$ is a scalar multiple of
the identity if and only if it equals zero. This is equivalent to
$\scal{\Im(\omega A)x_1,x_2}=\scal{\Im(\omega A)x_2,x_2}=0$. It
remains to invoke Lemma~\ref{l:cri}. \end{proof}

Thus the number of flat portions on the boundary of the numerical
range of the matrix (\ref{A}) coincides with the number of distinct
solutions $\omega$ of (\ref{cond1}), (\ref{cond2}) for which
$|\omega|=1$ and the ``if and only if"  conditions of Theorem
\ref{th:cri} are satisfied.

\section{\boldmath{$3 \times 3$}  matrices}

As was mentioned in the Introduction, the case
$n=2$ is trivial, and there is no need to consider companion $2\times
2$ matrices \eq{A2} \left[\begin{matrix}0 & 1\\ -a_0 &
-a_1\end{matrix}\right] \en separately. Nevertheless note that
conditions (\ref{cond1}), (\ref{cond2}) in this case amount to \[
a_0\omega^2=0, \quad \Re(a_1\omega)=0.\] They hold
if and only if $\abs{a_0}=1$ and $2\arg a_1-\arg a_0=\pi$ (the latter
condition being redundant when $a_1=0$).  These are exactly the
requirements for (\ref{A2}) to be normal, as it should be.

We move therefore to the case $n=3$.

\begin{theorem}\label{th:3}Let $A$ be a $3\times 3$ companion matrix:
\[ A=\left[ \begin{matrix}0 & 1 & 0 \\ 0 & 0 & 1\\ -a_0 & -a_1& -a_2\end{matrix}\right]. \]
Then $\partial W(A)$ contains a flat portion if and only if the equation
\eq{3cond1} a_0\omega^3+a_1\omega^2=1 \en  has a solution
$\omega$ with $\abs{\omega}=1$ in addition satisfying \eq{3cond2}
2\Re(a_2\omega)=\abs{a_0}^2-1,\en and the triple $a_0,a_1,a_2$
differs from \eq{cond3} a_0=-2\zeta^3,\ a_1=3\zeta^2\overline{w}, \
a_2=\frac{3}{2}\zeta w, \en where $\abs{\zeta}=1$ and $w$ is any
cube root of $1$.
\end{theorem}
\begin{proof}{\sl Necessity} of (\ref{3cond1}), (\ref{3cond2})  follows directly from Theorem~\ref{th:nec}. Indeed, (\ref{cond1})
for $n=3$ takes the form (\ref{3cond1}), while (\ref{gamma}) for $n=3$
and $j=2$ yields \[
\gamma_2=\frac{1}{2\sqrt{2}}\left(-1-a_0\omega^3+a_1\omega^2\right).
\] Taking (\ref{3cond1}) into consideration, we conclude further that $\gamma_2=-a_0\omega^3/\sqrt{2}$.
Based on this observation, (\ref{cond2}) with $n=3$ turns into
(\ref{3cond2}).

{\sl Sufficiency.} Conditions (\ref{3cond1}), (\ref{3cond2}), being the
$3\times 3$ version of (\ref{cond1}), (\ref{cond2}), guarantee that the
maximal eigenvalue of $\Re (\omega A)$ is not simple. By
Remark~\ref{Rem}, for $3\times 3$ matrices this implies the existence
of a flat portion on $\partial W(A)$ (with a slope
$\frac{\pi}{2}-\arg\omega$) provided that $A$ is unitarily irreducible. It
remains therefore to consider the case of unitarily reducible $A$.

According to Lemma~\ref{l:redu} in the case $n=3$,
the eigenvalues of a unitarily reducible $A$ are $\lambda_1=\eta\omega_1$,
$\lambda_2=\eta\omega_2$, $\lambda_3=\omega_3/\overline{\eta}$,
with $\omega_j$ corresponding to the three cube roots of unity (in no particular
order) and some non-zero $\eta$. Moreover, $A$ is then unitarily
similar to the orthogonal sum of a $2\times 2$ block $A_2$ with the
eigenvalues $\lambda_1,\lambda_2$ and the $1\times 1$ block
$A_1=\{\lambda_3\}$.

Letting $\abs{\eta}=r$ and $\arg\eta=\theta$, we therefore conclude
from Vieta's formulas that
\begin{align*} a_0 & = & -\lambda_1\lambda_2\lambda_3 & = & -\eta^2/\overline{\eta} & = & -re^{3i\theta}, \\
a_1 & = & \lambda_1\lambda_2+\lambda_1\lambda_3+\lambda_2\lambda_3
& =  & \eta^2\omega_1\omega_2+\eta(\omega_1+\omega_2)\omega_3/\overline{\eta}  & = &
(r^2-1)e^{2i\theta}\overline{\omega_3}, \\
a_2 & = & -(\lambda_1+\lambda_2+\lambda_3) & = & -\eta(\omega_1+\omega_2)-\omega_3/\overline{\eta} & = &
\frac{r^2-1}{r}e^{i\theta}\omega_3. \end{align*} If $r=1$, then
$a_1=a_2=0$, $\abs{a_0}=1$, so that (\ref{3cond2}) is a tautology
while (\ref{3cond1}) has three equidistant solutions $\omega$ on the
unit circle. The matrix $A$ is in this case unitary, and $W(A)$ has three
flat portions on the boundary.

On the other hand, (\ref{3cond2}) implies that
$2\abs{a_2}\geq\abs{a_0}^2-1$, that is, \[ 2\frac{\abs{r^2-1}}{r}\geq
\abs{r^2-1}. \] If $r\neq 1$, this is only possible when $r\leq 2$.

Due to the unitary similarity of $A$ and $A_1\oplus A_2$, the
numerical range $W(A)$ is the convex hull of $\lambda_3$ and
$W(A_2)$. The latter, in its turn, is the ellipse with the foci at
$\lambda_1,\lambda_2$ and the major axis of the length
\begin{multline*}
\sqrt{\tr(A_2^*A_2)-\abs{\lambda_1}^2-\abs{\lambda_2}^2+\abs{\lambda_1-\lambda_2}^2}=\\
\sqrt{\tr(A^*A)-\abs{\lambda_1}^2-\abs{\lambda_2}^2-\abs{\lambda_3}^2+\abs{\lambda_1-\lambda_2}^2}=\\
\sqrt{2+\abs{a_0}^2+\abs{a_1}^2+\abs{a_2}^2-2r^2-r^{-2}+3r^2}=\sqrt{1+r^2+r^4},
\end{multline*} while \[ \abs{\lambda_1-\lambda_3}+\abs{\lambda_2-\lambda_3}=\abs{r\omega_1-r^{-1}\omega_3}+
\abs{r\omega_2-r^{-1}\omega_3}=2\sqrt{1+r^2+r^{-2}}. \] But
\[ 2\sqrt{1+r^2+r^{-2}}>\sqrt{1+r^2+r^4} \text{ for } 0<r<2,  \] while for $r=2$ the equality is attained.
 Consequently, the point $\lambda_3$ lies outside the ellipse
$W(A_2)$, and their convex hull has two flat portions on the boundary,
unless $r=2$. It remains to observe that the case $r=2$ corresponds
exactly to the exception (\ref{cond3}).
\end{proof}
Note that Example~6 in \cite{GauWu07} is a particular case of
(\ref{cond3}) corresponding to $\zeta=w=1$.

Companion $3\times 3$ matrices with elliptical numerical ranges
were treated in \cite{Calb08}, based on the tests proposed in
\cite{KRS}. According to Kippenhahn's classification (see \cite{Ki08}),
for irreducible $3\times 3$ companion matrices $A$ not satisfying
conditions of \cite{Calb08} or our Theorem~\ref{th:3}, $W(A)$ has
ovular shape.

We remark that (\ref{3cond2}) is a tautology if $a_2=0$, $\abs{a_0}=1$,
it has no unimodular solutions if $2\abs{a_0}<\abs{\abs{a_0}^2-1}$, and
its (automatically unimodular) solutions are given by \[
\frac{\abs{a_0}^2-1\pm i\sqrt{4\abs{a_2}^2-(\abs{a_0}^2-1)^2}}{2a_2}
\] in the remaining case $0\neq 2\abs{a_2}\geq\abs{\abs{a_0}^2-1}$.
So, conditions (\ref{3cond1}), (\ref{3cond2}) can be recast as follows:
either
\begin{equation} \label{may135} a_2=0, \ |a_0|=1,\text{ and }  a_0\omega^3+ a_1\omega^2=1 \text{ for some unimodular }
\omega , \end{equation} or
\begin{equation}\label{may133} a_2\neq 0, \qquad 2\abs{a_2}\geq\abs{\abs{a_0}^2-1}\end{equation} and
\begin{multline}\label{may134}  a_0\left(\abs{a_0}^2-1+i\kappa
\sqrt{4\abs{a_2}^2-(\abs{a_0}^2-1)^2}\right)^3\\
+2a_1a_2\left(\abs{a_0}^2-1+i\kappa
\sqrt{4\abs{a_2}^2-(\abs{a_0}^2-1)^2}\right)^2=8a_2^3
\end{multline}
for some choice of $\kappa=\pm 1$.
\bigskip

{\sl Example 1.} Let
$a_0=2+i$, $a_1=-1-i$, and $a_2=2+3i$ so that we have: \eq{3flat}
A=\left[\begin{matrix} 0 & 1 & 0 \\ 0 & 0 & 1 \\
-2 -i & 1+i & -2-3i
\end{matrix}\right]. \en Then (\ref{3cond2}) holds with $\omega=1$, the exception (\ref{cond3})
does not hold, and (\ref{3cond1}) has only one unimodular solution:
$\omega_{1}=1$.  Thus, the matrix $A$ given by (\ref{3flat}) has one
(vertical) flat portion on the boundary of its numerical
range. As such, this $A$ is automatically unitarily irreducible.
The respective $W(A)$ is pictured in Figure~1
below\footnote{ All numerical ranges are plotted using
the program by C.~Cowen and E.~Harel, available at
\url{http://www.math.iupui.edu/~ccowen/Downloads/33NumRange.html}.}:

\begin{figure}[here]
\centering
\includegraphics[scale=.35]{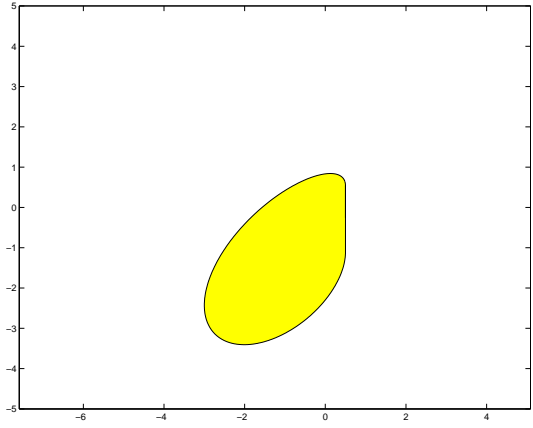}
\caption{Numerically calculated plot of the numerical range of $A$, as given by
(\ref{3flat}).}
\end{figure}

\newpage
\section{\boldmath{$4 \times 4$} matrices}\label{4}

In this section, we consider the case $n=4$, that is, \eq{com4} A=
\left[\begin{matrix}0 &
1 & 0 & 0\\
0 & 0 & 1 & 0 \\ 0 & 0 & 0 & 1\\ -a_0 & -a_1 & -a_2 &
-a_3\end{matrix}\right], \en  After some simple algebra, equations
(\ref{cond1}) and (\ref{cond2}) take the form
\begin{equation}\label{may136}
a_0\omega^4+\sqrt{2}a_1\omega^3+a_2\omega^2=1 \end{equation}
and
\begin{equation}\label{may137}
\sqrt{2} \Re (a_3\omega)=2|\gamma_2|^2+|\gamma_3|^2 -1,
\end{equation}
respectively. On the other hand, a computation shows that the
$\gamma_j$'s defined by (\ref{gamma}) with $n=4$ are given by the
formulas
$$  \gamma_2=\frac{1}{2\sqrt{2}} (-1-a_0\omega^4+a_2\omega^2), \quad  \gamma_3=\frac{1}{4}
(1-a_0\omega^4+\sqrt{2}a_1\omega_3-a_2\omega^2), $$ which
simplify further by using (\ref{may136}) to
\begin{equation}\label{may140} \gamma_2=-\frac{a_1\omega^3}{2}-\frac{a_0\omega^4}{\sqrt{2}},
\qquad \gamma_3=\frac{1}{\sqrt{2}} a_1\omega^3.
\end{equation}
Substitute (\ref{may140}) for $\gamma_2$ and $\gamma_3$ in the
right hand side of (\ref{may137}) to yield
\begin{equation}\label{may141}
{\sqrt{2}} \Re((a_3-a_0\overline{a_1})\omega)=
|a_0|^2+|a_1|^2 -1. \end{equation} So, the system of
equations (\ref{cond1}), (\ref{cond2}) is equivalent to the system
(\ref{may136}), (\ref{may141}).

Similarly to the situation for $n=3$, (\ref{may141}) is a tautology if
\eq{tau4} a_3=a_0\overline{a_1},
\quad\abs{a_0}^2+\abs{a_1}^2=1,\en it has no unimodular solutions if
$\sqrt{2}\abs{a_3-a_0\overline{a_1}}<\abs{\abs{a_0}^2+\abs{a_1}^2-1}$,
and its (automatically unimodular) solutions are given by \eq{oh}
\frac{|a_0|^2+|a_1|^2-1 \pm i \sqrt{2
\abs{a_3-a_0\overline{a_1}}^2-(|a_0|^2+|a_1|^2-1)^2}}{\sqrt{2}(a_3-a_0\overline{a_1})}\en
in the remaining case \[ 0\neq \abs{a_3-a_0\overline{a_1}}
\geq\frac{1}{\sqrt{2}}\abs{\abs{a_0}^2+\abs{a_1}^2-1}. \] We thus
obtain the following: \begin{corollary}\label{co:f4} Let $A$ be given by
{\em (\ref{com4})}.  Then for $W(A)$ to have a flat portion on the
boundary it is necessary that \[ \abs{a_3-a_0\overline{a_1}}
\geq\frac{1}{\sqrt{2}}\abs{\abs{a_0}^2+\abs{a_1}^2-1} \]  and {\em
(\ref{may136})}  has a unimodular solution $\omega$. Moreover, this
$\omega$ must coincide with one of the values given by {\em
(\ref{oh})}, unless {\em (\ref{tau4})} holds, and corresponds to the flat
portion (if any) with the slope $\frac{\pi}{2}-\arg\omega$.
\end{corollary}
This result is instrumental in establishing a peculiar gap in the
number of possible flat portions for $4\times 4$ companion matrices.
\begin{theorem}\label{th:no3}There are no $4\times 4$ companion matrices $A$ with $f(A)=3$.
\end{theorem}
\begin{proof}Let us first address the case when $A$ is unitarily reducible. According to Lemma~\ref{l:redu},
it is then
either unitary,  with the eigenvalues located in the vertices of a square
centered at the origin (in which case $f(A)=4$), or is unitarily similar to
the orthogonal sum of two unitarily irreducible blocks. If these blocks
are both $2\times 2$, then $W(A)$ is the convex hull of two ellipses
--- the construction that can a priori have 0, 2, or 4 flat portions (though the case
$f(A)=4$ does not materialize, as shown in \cite{GauWu07}) but not 1
or 3. Finally, if $A$ reduces to the orthogonal sum of a $1\times 1$ and
$3\times 3$ block, then the numerical range of the latter has no flat
portions on the boundary according to \cite[Theorem 2.5]{Gau10},
which leaves only options $f(A)=0,2$ possible.

Now let $A$ be unitarily irreducible. Applying Corollary~\ref{co:f4} we
see that $f(A)=3$ is only possible when (\ref{tau4}) holds and,
moreover, (\ref{may136}) has at least three distinct unimodular
solutions, say $u,v$ and $w$. We consider separately the cases
$a_0=0$ and $a_0\neq 0$.

{\sl Case 1.} $a_0=0$. The second equality in (\ref{tau4}) then implies
that $\abs{a_1}=1$. On the other hand, equation (\ref{may136}) in this
case has degree 3, and therefore $u,v,w$ are all its roots. By the Vieta
theorem, \[ uvw=\frac{1}{\sqrt{2}a_1},\] which is in contradiction with
the unimodularity of $u,v,w$.

{\sl Case 2.} $a_0\neq 0$. Then (\ref{may136}) has the fourth root, also
different from zero. Since the linear term  is missing in (\ref{may136}),
the inverses of the roots have zero sum. In other words, the fourth root
is \[ -\frac{1}{u^{-1}+v^{-1}+w^{-1}}=-1/\overline{z},\] where we have
denoted $z:=u+v+w$. Other  parts of the Vieta theorem mean that
\[ -uvw/\overline{z}=-1/a_0,\quad
z-1/\overline{z}=-\frac{\sqrt{2}a_1}{a_0}.\] Taking absolute values, we
obtain \[ \abs{a_0}=\abs{z}, \quad
\abs{a_1}=\abs{\abs{z}^2-1}/\sqrt{2}.\] When combined with the
second equality in (\ref{tau4}), this implies $\abs{z}=1$. Consequently,
$a_1=0$. Equation (\ref{may136}) is therefore biquadratic,  its roots
come in opposite pairs, and without loss of generality may be
relabeled as $\pm u, \pm v$.  By the same Vieta theorem,  \eq{uv}
a_0=-u^{-2}v^{-2}, \quad a_2= u^{-2}+v^{-2}. \en

While we have established that conditions of Corollary~\ref{co:f4} hold
for four distinct unimodular values of $\omega$, this does not
necessarily mean that four flat portions actually materialize. So, further
reasoning is needed in order to arrive at a contradiction. The first
equality in (\ref{tau4}) and the equality $a_1=0$ (proven earlier) imply
that $a_3=0$ as well. So, the characteristic polynomial (\ref{char}) in
our case also is biquadratic, and the eigenvalues of $A$ equal
$\pm\lambda_1,\pm\lambda_2$ with $\lambda_1^2,\lambda_2^2$
being the roots of the quadratic equation \[ \mu^2+a_2\mu+a_0=0. \]
The ratio of these roots is obviously a negative real number when
$a_2=0$. Supposing $a_2\neq 0$, on the other hand, we obtain
\begin{multline}\label{rat} \frac{\lambda_1^2}{\lambda_2^2}=
\frac{-a_2+\sqrt{a_2^2-4a_0}}{-a_2-\sqrt{a_2^2-4a_0}}=
\frac{\left(-a_2+\sqrt{a_2^2-4a_0}\right)^2}{4a_0} \\
=\frac{2a_2^2-4a_0-2a_2\sqrt{a_2^2-4a_0}}{4a_0}=-1+\frac{a_2^2-a_2\sqrt{a_2^2-4a_0}}{2a_0}\\
= -1+\frac{1-\sqrt{1-\frac{4a_0}{a_2^2}}}{2\frac{a_0}{a_2^2}}
=-1+\frac{\sqrt{1+2\left(\frac{-2a_0}{a_2^2}\right)}-1}{-2\frac{a_0}{a_2^2}}=
-1+\frac{\sqrt{1+2x}-1}{x},
\end{multline}  where $x=-2a_0/a_2^2$.  Using (\ref{uv}), \[  x=
\frac{2u^{-2}v^{-2}}{(u^{-2}+v^{-2})^2}=\frac{1}{1+\Re(u/v)^2}, \] and is
therefore a positive real number. Since for all such $x$,
$\sqrt{1+2x}<1+x$, expression (\ref{rat}) is again negative. So, the
eigenvalues $\pm\lambda_1, \pm\lambda_2$ of $A$ are located at
the vertices of a rhombus centered at the origin. According to
\cite{GauWu04}, this implies unitary reducibility of $A$ --- a
contradiction. Therefore, $f(A)=3$ is an impossibility in this case as
well, which concludes the proof.
\end{proof}

{\sl Example 2.}
We provide an explicit example of when $A$ is a
unitarily irreducible $4 \times 4$ companion matrix and $f(A)=2$. \\
Let $a_0=\frac{9+12i}{25}$, $a_1=\frac{2\sqrt{2}(7+i)}{25}$,
$a_2=-\frac{4(3+4i)}{25}$, and $a_3=\frac{6\sqrt{2}(1+i)}{25}$ so that
we have: \eq{flat2}
A=\left[\begin{matrix} 0 & 1 & 0 & 0 \\ 0 & 0 & 1 & 0 \\ 0 & 0 & 0 & 1 \\
-\frac{9+12i}{25} & -\frac{2\sqrt{2}(7+i)}{25} & \frac{4(3+4i)}{25} &
-\frac{6\sqrt{2}(1+i)}{25}
\end{matrix}\right]. \en
The eigenvalues of $A$, $\{0.6413 + 0.8475i, 0.6264 - 0.5578i, -1.0468 - 0.4290i, -0.5603 - 0.2000i\}$,
each have a different magnitude and therefore (Lemma~\ref{l:redu} again)
$A$ is not unitarily reducible. We also see that for the matrix $A$ given by (\ref{flat2}) that (\ref{tau4}) holds and (\ref{may136}) has two unimodular solutions,
$\omega_{1}=1$ and $\omega_{2}=i$, and two non-unimodular solutions, $\omega_{3}=-2+i$ and $\omega_{4}=\displaystyle -\frac{1}{3} - \frac{2i}{3}$.
Moreover, for $A_{1}=\omega_{1} A =A$, $\Re A_{1}$ has two linearly independent
eigenvectors, $f_{1}=\displaystyle \left[ \frac{\sqrt{2}}{2}, 1,
\frac{\sqrt{2}}{2}, 0 \right]^{T}$ and $f_{2}=\displaystyle \left[
\frac{\sqrt{2}(-23+14i)}{25}, \frac{-37+16i}{25}, 0, 1 \right]^{T}, $
corresponding to the maximal eigenvalue of $\displaystyle
\frac{\sqrt{2}}{2}$. Computing the scalar product $\scal{(\Im A_{1})f_{1}, f_{2}}=\displaystyle
\frac{\sqrt{2}(-7+24i)}{25} \neq 0$ we see that indeed $W(A)$
has a vertical flat portion on the boundary.
Similarly, for
$A_{2}=\omega_{2} A = iA$, $\Re A_{2}$ has two linearly independent
eigenvectors, $g_{1}=\displaystyle \left[ -1, \sqrt{2}i, 1, 0 \right]^{T}$
and $g_{2}=\displaystyle \left[ \frac{\sqrt{2}(-2+11i)}{25},
\frac{13+16i}{25}, 0, 1 \right]^{T}$, corresponding to the maximal
eigenvalue of $\displaystyle \frac{\sqrt{2}}{2}$, while $\scal{(\Im A_{2})g_{1},g_{2}} = \displaystyle
\frac{36-2i}{25}\neq 0$. Therefore,  $W(A)$ also has a
horizontal flat portion on its boundary.

Thus, the matrix $A$ given by (\ref{flat2}) has two flat portions on the
boundary of $W(A)$, as shown in Figure 2.

\begin{figure}[here]
\centering
\includegraphics[scale=.35]{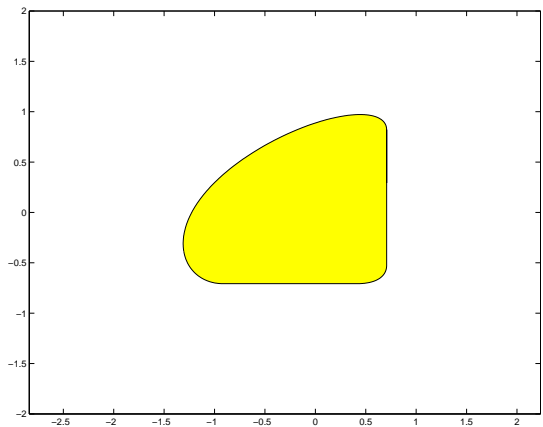}
\caption{Numerically calculated plot of the numerical range of $A$, as given by
(\ref{flat2}).}
\end{figure}

\newpage

{\sl Example 3.}We provide an explicit example of when $A$ is a
unitarily irreducible $4 \times 4$ companion matrix and $f(A)=1$. Let
$a_0=0$, $a_1=1$, $a_2=1-\sqrt{2}$, and $a_3=0$ so that we have: \eq{flat1}
A=\left[\begin{matrix} 0 & 1 & 0 & 0 \\ 0 & 0 & 1 & 0 \\ 0 & 0 & 0 & 1 \\
0 & -1 & -1+\sqrt{2} & 0
\end{matrix}\right]. \en Then (\ref{tau4}) holds and (\ref{may136}) has only one unimodular solution:
$\omega_{1}=1$. Moreover, for $A_{1}=\omega_{1} A =A$, $\Re A_{1}$
has two linearly independent eigenvectors, $f_{1}= [1, \sqrt{2}, 1, 0]^{T}$ and $f_{2}=[-1, -\sqrt{2}, 0, 1]^{T}$,
corresponding to the maximal eigenvalue of $\sqrt{2}/2$. Computing the
scalar product $\scal{(\Im A_{1})
f_{1}, f_{2}} = \displaystyle \frac{2+\sqrt{2}}{2} i$ we see indeed that $W(A)$ has a vertical flat
portion on the boundary.

Thus, the matrix $A$ given by (\ref{flat1}) has one flat portion on the
boundary of $W(A)$, as shown in Figure 3.

\begin{figure}[here]
\centering
\includegraphics[scale=.35]{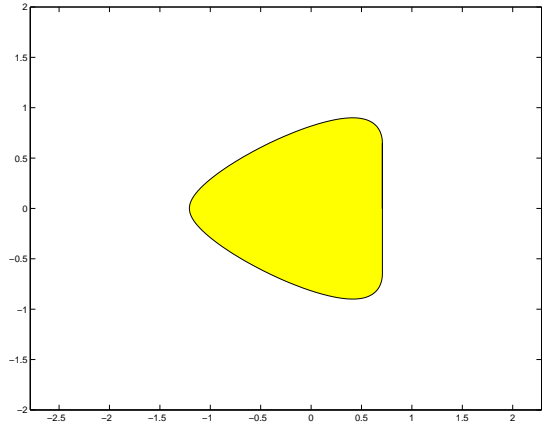}
\caption{Numerically calculated plot of the numerical range of $A$, as given by
(\ref{flat1}).}
\end{figure}

\newpage
Note that having exactly one flat portion on $W(A)$ implies unitary irreducibility of the matrix (\ref{flat1}), as was shown in the proof of Theorem~\ref{th:no3} (see the first paragraph there).

Finally, let us provide an example of a $4\times 4$ matrix satisfying conditions of Corollary~\ref{c:zeros}, and thus unitarily irreducible with no flat portions on the boundary of its numerical range.

{\sl Example 4.}
Let \eq{flat0}
A=\left[\begin{matrix} 0 & 1 & 0 & 0 \\ 0 & 0 & 1 & 0 \\ 0 & 0 & 0 & 1 \\
0 & 0 & 0 & -2 \end{matrix}\right], \en
that is,
$a_0=a_1=a_2=0$ and $a_3=2$. The numerical range of this matrix is given in Figure 4.

\begin{figure}[here]
\centering
\includegraphics[scale=.35]{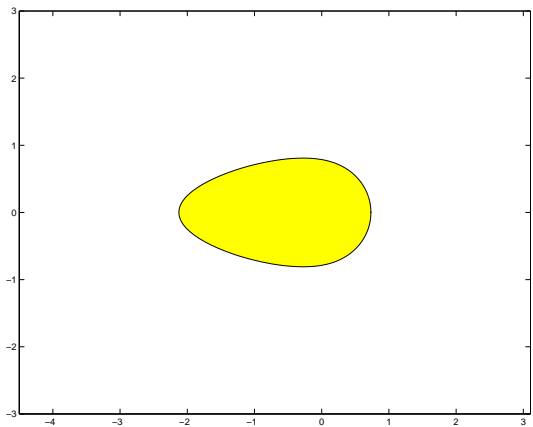}
\caption{Numerically calculated plot of the numerical range of $A$, as given by (\ref{flat0}).}
\end{figure}

\newpage
Gathering information from Corollary~\ref{c:zeros} (or Example 4), Theorem~\ref{th:no3} and
Examples 2--3 we arrive to our final conclusion.
\begin{theorem}For a $4\times 4$ unitarily irreducible companion matrix $A$ the complete list
of admissible values of $f(A)$ is $\{0,1,2\}$.\end{theorem}

\providecommand{\bysame}{\leavevmode\hbox to3em{\hrulefill}\thinspace}
\providecommand{\MR}{\relax\ifhmode\unskip\space\fi MR }
\providecommand{\MRhref}[2]{%
  \href{http://www.ams.org/mathscinet-getitem?mr=#1}{#2}
}
\providecommand{\href}[2]{#2}

\end{document}